\documentclass[12pt,reqno]{amsart}

\usepackage{verbatim, amssymb, enumitem, mathtools}

\usepackage[breaklinks=true,colorlinks=true,linkcolor=blue,citecolor=red,urlcolor=blue,psdextra,pdfencoding=auto]{hyperref}
\allowdisplaybreaks

\renewcommand \a{\alpha}

\newcommand \la{\lambda}

\newcommand \br{\mathbb{R}}
\newcommand \bc{\mathbb{C}}

\newcommand \Z{\mathbb{Z}}
\newcommand \F{\mathbb{F}}

\newcommand \bh{\mathbb{H}}
\newcommand \bo{\mathbb{O}}

\newcommand \rk{\operatorname{rk}}

\newcommand \Span{\operatorname{Span}}
\newcommand \Tr{\operatorname{Tr}}

\newcommand \SO{\mathrm{SO}}
\newcommand \Sp{\mathrm{Sp}}
\newcommand \SU{\mathrm{SU}}

\newcommand \Ff{\mathrm{F}_4}
\newcommand \Spin{\mathrm{Spin}}

\newcommand \cB{\mathcal{B}}
\newcommand \cC{\mathcal{C}}

\newcommand \cV{\mathcal{V}}

\newcommand \cK{\mathcal{K}}

\newcommand \sK{\mathsf{K}}
\newcommand \sS{\mathsf{S}}

\newcommand \sT{\mathsf{T}}

\newcommand\ag{\mathfrak a}

\newcommand\g{\mathfrak g}

\newcommand\h{\mathfrak h}

\newcommand\m{\mathfrak m}
\newcommand \so{\mathfrak{so}}
\newcommand \spg{\mathfrak{sp}}

\newcommand \ug{\mathfrak{u}}
\newcommand \su{\mathfrak{su}}

\newcommand \ir{\mathrm{i}}
\newcommand \jr{\mathrm{j}}
\newcommand \kr{\mathrm{k}}

\newcommand \Sym{\operatorname{Sym}}

\newcommand \diag{\operatorname{diag}}

\newcommand \<{\langle}
\renewcommand \>{\rangle}
\newcommand \ip{\<\cdot,\cdot\>}

\newtheorem{theorem}{Theorem}
\newtheorem*{theorem*}{Theorem}

\newtheorem*{corollary*}{Corollary}
\newtheorem*{conj*}{Conjecture}
\newtheorem{lemma}{Lemma}
\newtheorem{proposition}{Proposition}
\newtheorem*{prop*}{Proposition}

\theoremstyle{definition}

\newtheorem*{definition*}{Definition}

\theoremstyle{remark}
\newtheorem{remark}{Remark}

\newtheorem*{notation*}{Notation}
\newtheorem*{algorithm*}{Algorithm}
\newtheorem*{example*}{Example}

\begin{document}
	
	\title[Quadratic Killing tensors on classical Lie groups are decomposable]{Quadratic Killing tensors on classical Lie groups are decomposable}
	
	\author{Vladimir Matveev}
	\address{Fakult\"{a}t f\"{u}r Mathematik und Informatik\\ Friedrich-Schiller-Universit\"{a}t\\ 07737 Jena, Germany}
	\email{vladimir.matveev@uni-jena.de}
	
	\author{An Ky Nguyen}
	\address{Department of Mathematical and Physical Sciences\\ La Trobe University\\ Melbourne, Victoria 3086, Australia}
	\email{AnKy.Nguyen@latrobe.edu.au}
	
	\author{Yuri Nikolayevsky}
	\address{Department of Mathematical and Physical Sciences\\ La Trobe University\\ Melbourne, Victoria 3086, Australia}
	\email{y.nikolayevsky@latrobe.edu.au}
	
	\thanks{The first and the third named authors were partially supported by ARC Discovery grant DP210100951. The first named author was partially supported by the DFG project 529233771.}
	
	% 53B20(1973-now)Local Riemannian geometry
	% 53D25(2000-now)Geodesic flows in symplectic geometry and contact geometry
	% 53C35(1973-now)Differential geometry of symmetric spaces
	% 22E46(1980-now)Semisimple Lie groups and their representations
	% 70H06(2000-now)Completely integrable systems and methods of integration for problems in Hamiltonian and Lagrangian mechanics
	

	\begin{abstract}
		A Killing tensor field on a Riemannian manifold $(M,g)$ is a covariant symmetric tensor field whose contraction with the velocity vector along a geodesic produces a homogeneous polynomial first integral of the geodesic flow. Such a tensor is called \emph{decomposable} if it lies in the subalgebra generated by Killing vector fields; equivalently, the corresponding polynomial integral is then a polynomial in the linear integrals coming from infinitesimal isometries. On spaces of constant sectional curvature and on the complex projective space, every Killing tensor field is decomposable. By contrast, the quaternionic projective spaces and the Cayley projective plane admit indecomposable quadratic Killing tensor fields. We prove that every quadratic Killing tensor field on the compact classical Lie groups $\SO(n)$, $\Spin(n)$, $\SU(n)$ and $\Sp(n)$, equipped with a bi-invariant Riemannian metric, is decomposable; equivalently, every quadratic first integral of the geodesic flow on these groups is a quadratic polynomial in the linear first integrals. 
	\end{abstract}

	\maketitle
	
	\section{Introduction}
	\label{s:intro}
	
	One of the most effective ways to probe the geometry of a Riemannian manifold is through the conserved quantities of its geodesic flow. Killing vector fields give the first and most classical such conservation laws: they are infinitesimal isometries, and by Noether's principle they produce first integrals which are linear in the velocities.  Killing tensor fields are the higher-degree analogue. They produce polynomial first integrals of the geodesic flow and therefore detect higher-order symmetries of the metric, symmetries which need not come from isometries themselves.
	
	This distinction is central both in differential geometry and in geometric mechanics. A quadratic Killing tensor may either be built from Killing vector fields, in which case the associated quadratic integral is algebraically forced by ordinary symmetries, or it may be irreducible, in which case it represents a genuinely hidden symmetry. The latter phenomenon is familiar in mathematical physics: the Carter constant in Kerr geometry, and the Killing tensor interpretation developed by Walker and Penrose, are classical examples showing that quadratic first integrals can encode structure invisible to the isometry algebra alone~\cite{Car,WP}. The problem studied in this paper is the corresponding structural question in Riemannian symmetric geometry: for which symmetric spaces do all quadratic hidden symmetries in fact reduce to visible ones?
	
	Let $K=K(x)_{i_1\dots i_d}$ be a covariant symmetric tensor field of rank $d\geq 1$ on a Riemannian manifold $(M,g)$. It is called a \emph{Killing tensor field} if it satisfies the Killing equation
	\begin{equation*}
		K_{(i_1\dots i_d,j)}=0,
	\end{equation*}
	where the comma denotes covariant differentiation and the parentheses denote symmetrisation over all indices. This equation is equivalent to the assertion that
	\[
	\xi\in T_xM
	\longmapsto
	K(x)_{i_1\dots i_d}\xi^{i_1}\cdots \xi^{i_d}
	\]
	is a homogeneous polynomial first integral of the geodesic flow. In other words, for every naturally parameterised geodesic $s\mapsto \gamma(s)$, the quantity
	\[
	K(\gamma(s))_{i_1\dots i_d}
	\dot\gamma(s)^{i_1}\cdots \dot\gamma(s)^{i_d}
	\]
	is constant. When $d=1$, one recovers Killing vector fields after raising an index by the metric. For $d\geq 2$, however, there is no comparably simple geometric description in general.
	
	The space of all Killing tensor fields on $(M,g)$ is an associative, commutative, graded algebra $\sK(M)$ under symmetric tensor product. The rank-one elements generate a graded subalgebra $\sS(M)\subseteq \sK(M)$. The elements of $\sS(M)$ are called \emph{decomposable}.  Equivalently, a Killing tensor is decomposable if the corresponding polynomial integral of the geodesic flow is a polynomial in the linear integrals arising from Killing vector fields. Thus decomposable tensors reflect no new conservation laws beyond those already imposed by the isometry group, while indecomposable tensors represent genuinely higher-order hidden symmetries.
	
	In general, Killing tensor fields of rank at least two need not be decomposable, even modulo the metric tensor. The classical two-dimensional Liouville metrics already exhibit non-trivial quadratic Killing tensors although their isometry groups are generically trivial. At the opposite extreme, strong symmetry often forces decomposability. Every Killing tensor field on a Riemannian space of constant sectional curvature is decomposable~\cite{Tho,ST1,Tak}. The same is true on complex projective space with its Fubini--Study metric~\cite[Corollary~5]{East}, \cite[Theorem~2.2]{ST2}. These positive results suggest that, in sufficiently rigid geometries, all polynomial first integrals might be generated by infinitesimal isometries.
	
	Symmetric spaces provide the natural testing ground for this expectation. They have large isometry groups and strong curvature identities, yet their geometry is rich enough to allow non-obvious higher-order phenomena. Motivated by the known positive cases and by the role of symmetric spaces in finite-dimensional integrable systems, \cite[Question~3.9]{BMMT} asked whether every Killing tensor field on a symmetric space is decomposable (the question is originally attributed  to Milson and Fels).
	
	The answer is already subtle in rank one.  For spheres and complex projective spaces the answer is positive, as above. Somewhat unexpectedly, the remaining compact rank-one symmetric spaces behave differently at the quadratic level: Matveev and Nikolayevsky showed that the quaternionic projective spaces $\bh P^n=\Sp(n+1)/(\Sp(n)\Sp(1))$, $n\geq 3$, and the Cayley projective plane $\bo P^2=\Ff/\Spin(9)$ admit quadratic Killing tensor fields which are not quadratic forms in Killing vector fields~\cite{MN1}. Very recently, Eastwood and Leistner proved the existence of indecomposable quadratic Killing tensor fields on the exceptional symmetric space $\mathrm{E}_6/\Ff$ of rank $2$~\cite[Theorem~16]{EL}. For all these spaces, the space of quadratic Killing tensor fields is spanned by the decomposable and the explicitly constructed indecomposable ones (\cite[Theorem~4]{MN2} and~\cite[Theorem~17]{EL}, respectively).
	
The higher-rank case is much less understood. The main result of this paper is that for compact, classical Lie groups, no such irreducible quadratic hidden symmetries exist.
	\begin{theorem} \label{t:SOSUSp}
		All quadratic Killing tensor fields on the classical Lie groups $\SO(n)$ \emph{(}and $\Spin(n)$\emph{)}, $\SU(n)$ and $\Sp(n)$, equipped with a bi-invariant metric, are decomposable, that is, are quadratic forms on the space of Killing vector fields.
	\end{theorem}
	
	Equivalently, every quadratic first integral of the geodesic flow on these groups is a quadratic polynomial in the linear first integrals generated by infinitesimal isometries. By~\cite[Theorem~2]{MN2}, the same conclusion holds for the corresponding non-compact dual symmetric spaces. 
	
We now describe the proof strategy. We start with Proposition~\ref{p:topslot2} which reduces the question of determining quadratic Killing tensors on a given symmetric space to a purely linear-algebraic one. Further in Subsection~\ref{ss:topslot} we prove several important structural results given in Lemma~\ref{l:nosyz}, Lemma~\ref{l:fourcommute} and Lemma~\ref{l:threecommute}; they hold for all compact irreducible symmetric spaces, not only for classical groups with bi-invariant metric. Using these results, together with the totally geodesic inclusions $\Spin(n-1)\subset\Spin(n), \, \SU(n-1)\subset\SU(n)$ and $\Sp(n-1)\subset\Sp(n)$, we give the induction proof of Theorem~\ref{t:SOSUSp} in Subsection~\ref{ss:induction}. For all three series of the groups, the induction step starts working from $n=6$. The cases $n \le 5$ have to be treated manually. The proof for these cases is given in Subsection~\ref{ss:sage} and is computer aided. We work with \texttt{SageMath} and use finite field arithmetic and sparse matrix packages. The corresponding worksheets are supplied as ancillary material.

	\section{Preliminaries}
	\label{s:prel}
	
	\subsection{Quadratic Killing tensor fields on symmetric spaces}
	\label{ss:topslot}
	
	For general theory of symmetric spaces we refer the reader to~\cite{Hel,Wolf}.
	
	Let $M=G/H$ be a simply connected, connected, irreducible Riemannian globally symmetric space, where $G$ is the connected component of the full isometry group of $M$, and $H$ is the isotropy subgroup at a point $o \in M$. We have the corresponding decomposition $\g = \h \oplus \m$ at the level of Lie algebras, where $\g$ and $\h$ are the Lie algebras of $G$ and $H$, respectively, and $\m$ is the $\h$-submodule of $\g$ orthogonal to $\h$ relative to the Killing form of $\g$. Then there is a standard identification of $\m$ with the tangent space $T_oM$. The curvature tensor of $M$ at $o$ is given by $R(X,Y)Z = -[[X,Y],Z]$. The restriction of the Killing form of $\g$ to $\h$ gives an inner product on $\h$; minus the restriction of the Killing form of $\g$ to $\m$ gives an inner product on $\m$. Both these inner products are $\h$-invariant, and we denote both of them by $\ip$. Up to scaling, the inner product induced by the metric of $M$ on $T_oM$ coincides with $\ip$.
	
	The following necessary and sufficient condition of decomposability of quadratic Killing tensors is obtained in~\cite[Theorem~3, Remark~2]{MN2}.
	
	\begin{proposition} \label{p:topslot2}
		Let $M= G/H$ be an irreducible, simply connected globally symmetric space. Denote by $\sT(\m)$ the linear space of (constant) tensors $K$ of type $(0,4)$ on $\m$ which satisfy the following equations, for all $X, Y, Z, P \in \m$:
		\begin{gather}
			K(X,Y,Z,P) = K(Y,X,Z,P) = K(X,Y,P,Z), \quad \sigma_{X,Y,Z} \big(K(X,Y,Z,P)\big)=0, \label{eq:K2c}\\
			K(X,X,P,R(X,P)P)=K(P,P,X,R(P,X)X), \label{eq:K21}\\
			K(X,X,R(X,P)P,R(X,P)P)=K(P,P,R(P,X)X,R(P,X)X), \label{eq:K22}
		\end{gather}
		where $\sigma_{X,Y,Z}$ denotes the sum over cyclic permutations of $(X,Y,Z)$.
		
		Every quadratic Killing tensor field on $M$ is decomposable if and only if every element $K \in \sT(\m)$ is decomposable in the sense that there exists a quadratic form $Q$ on $\h$ such that $K=K_Q$, where
		\begin{equation} \label{eq:KQ}
			K_Q(X,X,P,P) = Q([X,P],[X,P]),
		\end{equation}
		for all $X, P \in \m$.
	\end{proposition}
    \begin{remark} \label{rem:Spin4}
      For the induction argument in Subsection~\ref{ss:induction} we will need the fact that the claim of Proposition~\ref{p:topslot2} also holds for the \emph{reducible} symmetric space $\Spin(4)$. As any Killing tensor field on $\Spin(4)$ is decomposable (see Remark~\ref{rem:low}), we need to show that any element $K \in \sT(\so(4))$ has the form $K_Q$ given by~\eqref{eq:KQ}. But by~\cite[Proposition~2]{MN2}, any element $K \in \sT(\so(4))$ defines a top slot quadratic Killing tensor field on $\Spin(4)$, and then the claim follows from~\cite[Remark~3]{MN2}.
    \end{remark}
	\begin{remark} \label{rem:dim}
		By the algebraic symmetries~\eqref{eq:K2c}, any element $K \in \sT(\m)$ is uniquely determined by its values $K(X,X,P,P)$. The space of tensors on $\m$ satisfying~\eqref{eq:K2c} is isomorphic to the space of algebraic curvature tensors on $\m$ as a module of the permutation group on four elements and has dimension $\frac{1}{12} N^2(N^2-1)$, where $N = \dim \m$~\cite[Remark~3]{MN2}. The space $\sT(\m)$ sits inside this ambient space as the subspace further cut out by~\eqref{eq:K21} and~\eqref{eq:K22}.
		
		One easily verifies that the tensors $K_Q$ defined by~\eqref{eq:KQ} indeed satisfy~\eqref{eq:K2c}, \eqref{eq:K21} and~\eqref{eq:K22} using the fact that for a symmetric space, we always have $[X,R(X,P)P]+[P,R(P,X)X]=0$. This equation is equivalent to $[X,[[X,P],P]]+[P,[[P,X],X]]=0$ which follows from the Jacobi identity.
	\end{remark}
	
	Polarising~\eqref{eq:KQ} we obtain
	\begin{equation}\label{eq:KQpolar}
		K_Q(X,Y,Z,P) = \frac12 (Q([X,P],[Y,Z]) + Q([X,Z],[Y,P])),
	\end{equation}
	for all $X,Y,Z,P \in \m$.
	
We prove the following fact similar to~\cite[Corollary~3]{EL}.
	\begin{lemma} \label{l:nosyz}
		Let $M$ be a compact, simply connected globally symmetric space. In each of the following cases:
        \begin{enumerate}[label=\emph{(\roman*)},ref=\roman*]
          \item \label{it:nosyzirred}
          $M$ is irreducible of non-constant curvature,

          \item \label{it:nosyzSpin}
          $M$ is isometric to $\Spin(3)$ or $\Spin(4)$ with a bi-invariant metric,
        \end{enumerate}
        the linear map $Q \mapsto K_Q$ defined by~\eqref{eq:KQ} is injective.
	\end{lemma}
	\begin{proof} 
		We start with assertion~\eqref{it:nosyzirred}. A quadratic form $Q$ on $\h$ can be extended to a quadratic form $Q'$ on $\g$ in the obvious way, by setting $Q'(X_1+A_1,X_2+A_2) = Q(A_1,A_2)$, for $X_1, X_2 \in \m, \, A_1,A_2 \in \h$. The quadratic form $Q'$ defines a decomposable quadratic Killing tensor field $\cK$ on $M$ which is top slot at the point $o \in M$. By~\cite[Theorem~3]{MN2}, there is a linear isomorphism between the space of quadratic Killing tensor fields which are top slot at $o \in M$ and the space $\sT(\m)$, under which the tensor field $\cK$ corresponds to the tensor $K_Q \in \sT(\m)$. Therefore if $K_Q = 0$, then $\cK = 0$, and so $Q'$ lies in the kernel of the homomorphism~\cite[Equation~(43)]{EL} mapping quadratic forms on Killing vector fields on $M$ to quadratic Killing tensor fields on $M$. If $M$ is not a compact simple Lie group with a bi-invariant metric, this homomorphism is injective by~\cite[Corollary~3]{EL}, and so we must have $Q' = 0$, from which $Q=0$, as required.

If $M$ is  a compact simple Lie group $H$ with a bi-invariant metric, then $\g = \h_1 \oplus \h_2$ (direct sum of ideals), where we identify $\h_1$ with $\h$ and $\h_2$ with $\m$. We take the inner product on $\g$ in such a way that $\h_1$ and $\h_2$ are orthogonal, and whose restriction $\ip_i$ to each copy $\h_i$ of $\h$ is the same negative multiple of the Killing form of $\h$. The kernel of the homomorphism~\cite[Equation~(43)]{EL} is $1$-dimensional and is spanned by the quadratic form $\ip_1 - \ip_2$. But this quadratic form is not zero on $\m = \h_2$, and so is not equal to the quadratic form $Q'$ defined in the previous paragraph, for any quadratic form $Q$ on $\h = \h_1$. This completes the proof of assertion~\eqref{it:nosyzirred}.
	
Assertion~\eqref{it:nosyzSpin} is easy. Indeed, by~\eqref{eq:KQ} it is sufficient to show that for $n=3,4$, a quadratic form $Q$ on $\so(n)$ is zero, provided $Q(Y,Y)=0$ for all $Y$ lying in the image of the bracket map $(X,P) \mapsto [X,P]$. But in both cases, that image is the whole algebra $\so(n)$.
\end{proof}

We record the following useful fact.
	\begin{lemma} \label{l:fourcommute}
		Let $M= G/H$ be an irreducible symmetric space. Then for any $K \in \sT(\m)$ we have $K(X_1,X_2,X_3,X_4) = 0$ provided $[\Span(X_1,X_2),\Span(X_3,X_4)] = 0$.
	\end{lemma}
	\begin{proof}
		From~\cite[Corollary~2]{MN2}, applied to quadratic Killing tensors, together with Proposition~\ref{p:topslot2}, it follows that for any $K \in \sT(\m)$ we have $K(X,X,P,P) = 0$ when $[X, P] = 0$. Now the claim follows by polarisation using the symmetries~\eqref{eq:K2c}.
	\end{proof}
		
We prove the following stronger statement.

	\begin{lemma} \label{l:threecommute}
		Let $M$ be a compact, irreducible globally symmetric space. Then for any $K \in \sT(\m)$, we have $K(X_1,X_2,X_3,X_4) = 0$ provided $[X_1,X_3]=[X_2,X_3] =0$.
	\end{lemma}
	We note that by the algebraic symmetries~\eqref{eq:K2c}, the same is true for any quadruple $(X_1,X_2,X_3,X_4)$ with any one of the elements from either the first or the second pair commuting with both elements of the other pair. We also note that any tensor $K_Q$ has the claimed property, as can be seen directly from~\eqref{eq:KQpolar}.

	\begin{proof}
		Let $K \in \sT(\m)$. By the symmetry of $K$ in the first two arguments~\eqref{eq:K2c}, it is sufficient to prove that $K(X_1,X_1,X_3,X_4) = 0$ assuming $[X_1,X_3]=0$. As any two commuting elements of $\m$ lie in a Cartan subspace, it is sufficient to show that $K(X_1,X_2,X_3,X_4) = 0$ assuming $X_1,X_2$ and $X_3$ lie in some Cartan subspace $\ag \subset \m$, or equivalently, to show that $K(\ag,\ag,\ag,\m) = 0$. We note that if $\rk M =1$, the Cartan subspace $\ag$ is $1$-dimensional and the claim follows from the last equation of~\eqref{eq:K2c}. We will therefore assume that $\rk M \ge 2$. % helgason, lemma V.6.3; called max abelian subspace; check in Wolf: ok, Th 8.6.1
		
		In equation~\eqref{eq:K22}, take $X=t_1Y_1+t_2Y_2+t_3Y_3$, with $[Y_3,P] = [Y_3,Y_1]=0$. Then
		\[
		\begin{aligned}
			R(X,P)P &= -[[t_1Y_1+t_2Y_2,P],P],\\
			R(P,X)X &= -[[P,t_1Y_1+t_2Y_2],t_1Y_1+t_2Y_2]
			-t_2t_3[[P,Y_2],Y_3].
		\end{aligned}
		\]
		Therefore the coefficient of $t_1t_2t_3^2$ in the resulting equation gives
		\[
		K(Y_3,Y_3, [[Y_1,P],P], [[Y_2,P],P]) = 0.
		\]
		Now take a Cartan subspace $\ag \subset \m$, let $P, Y_3$ lie in $\ag$, and let $Y_1$ be an element of a root space $\m_\a$ such that $\a(Y_3) = 0$, so that $[Y_3,P] = [Y_3,Y_1]=0$, by our choice. Then $[[Y_1,P],P] = -\a(P)^2 Y_1$, and we obtain $K(Y_3,Y_3,Y_1, [[Y_2,P],P]) = 0$. Note that we have imposed no assumptions on $Y_2$. As $\Span([[Y_2,P],P] \, | \, Y_2 \in \m, \, P \in \ag) = \ag^\perp$, we obtain $K(Y_3,Y_3,Y_1, \ag^\perp) = 0$. But $K(Y_3,Y_3,Y_1, \ag) = 0$ by Lemma~\ref{l:fourcommute}. Hence denoting $Y=Y_3$ and $Z = Y_1$ we obtain:
		\begin{equation}\label{eq:threeag}
			K(Y,Y,Z,\m) = 0,
		\end{equation}
% co-root
		for any two elements $Y,Z \in \m$ such that $[Y,Z] =0$ and such that there exists a Cartan subspace $\ag \ni Y$ which is orthogonal to $Z$. Suppose $H_\a \in \ag$ is a dual root of some root $\a$. Then for any nonzero $X_\a \in \m_\a$, the subspace $\ag' = (\ag \cap H_\a^\perp) \oplus \br \, X_\a \subset \m$ is abelian of dimension $\rk M$, and hence is a Cartan subspace orthogonal to $H_\a$. From~\eqref{eq:threeag} it follows that
		\begin{equation}\label{eq:threeagroot}
			K(X,Y,H_\a,\m) = 0,
		\end{equation}
for any dual root $H_\a \in \ag$ and any $X, Y \in \ag$ such that $X,Y \perp H_\a$.

We need to show that~\eqref{eq:threeagroot} implies that $K(\ag,\ag,\ag,\m) = 0$. The proof of this fact only requires working on $\ag$ and only depends on the dual root system $\Delta^* \subset \ag$. We note that if the dual root systems $\Delta^*_1$ and  $\Delta^*_2$ have the same rank and  $\Delta^*_2 \subset \Delta^*_1$, it suffices to establish the required implication only for  $\Delta^*_2$; for $\Delta^*_1$, we simply have more equations in~\eqref{eq:threeagroot}. We have the following inclusions between the root systems of irreducible symmetric spaces of the same rank $r \ge 2$: $\mathrm{D}_r \subset \mathrm{B}_r, \mathrm{C}_r, \mathrm{BC}_r; \; \mathrm{D}_4 \subset \mathrm{F}_4; \;  \mathrm{D}_8 \subset \mathrm{E}_8; \; \mathrm{A}_7 \subset \mathrm{E}_7$ and $\mathrm{A}_2 \subset \mathrm{G}_2$. It follows that we need to consider only the root systems $\mathrm{A}_r, \, r \ge 2, \; \mathrm{E}_6$ and $\mathrm{D}_r, \, r=2$ or $r \ge 4$ (as $\mathrm{D}_3 = \mathrm{A}_3$); note that $\mathrm{D}_2$ is reducible, but we need it for $\mathrm{B}_2, \mathrm{C}_2$ and $\mathrm{BC}_2$.

For $\mathrm{D}_2$, the claim is immediate. In $\ag = \br^2$ with an orthonormal basis $\{e_1,e_2\}$, we have $\Delta^* = \{\pm e_1, \pm e_2\}$. Then from~\eqref{eq:threeagroot} we obtain $K(e_1,e_1,e_2,\m) = K(e_2,e_2,e_1,\m) = 0$, and from the last equation of~\eqref{eq:K2c} it follows that $K(\ag,\ag,\ag,\m) = 0$.

For $\mathrm{A}_r, \, r \ge 2$, we consider the Euclidean space $\br^{r+1}$ with an orthonormal basis $\{e_i\}$. Then $\ag = (\sum_{i=1}^{r+1}e_i)^\perp$ and $\Delta^* = \{u_{ij} = e_i-e_j \, | \, 1 \le i \ne j \le r+1\}$. As the elements $u_{1i}, \, i \ge 2$, form a basis for $\ag$, it is sufficient to show that $K(u_{12},u_{12},u_{13},X) = K(u_{12},u_{13},u_{14},X)= 0$, for all $X \in \m$. To prove the first equation we notice that $K(2u_{13}-u_{12},2u_{13}-u_{12},u_{12},X) = 0$ by~\eqref{eq:threeagroot}, and so by~\eqref{eq:K2c} we get $4K(u_{13},u_{13},u_{12},X) + 2K(u_{12},u_{12},u_{13},X) = 0$. Swapping the subscripts $2$ and $3$ we obtain $K(u_{12},u_{12},u_{13},X) = K(u_{13},u_{13},u_{12},X) = 0$, as required. For the second equation, we have $K(u_{12}-u_{13}, u_{12}-u_{13}, u_{14},X) = 0$ from~\eqref{eq:threeagroot}, and as $K(u_{12}, u_{12}, u_{14},X) = K(u_{13}, u_{13}, u_{14},X) = 0$ by the previous argument, we obtain $K(u_{12},u_{13},u_{14},X)= 0$.

For both $\mathrm{D}_r, \, r \ge 4$, and $\mathrm{E}_6$ we use the following argument. Take a subset $S \subset \Delta^*$ which spans $\ag$. It is sufficient to show that $K(X,Y,H_\a,\m) = 0$, for any $H_\a \in S$ and any $X, Y \in \ag$. If $X, Y \perp H_\a$, this follows from~\eqref{eq:threeagroot}, while if $X=Y=H_\a$, from the last equation of~\eqref{eq:K2c}. It remains to show that $K(X,H_\a,H_\a,\m) = 0$ for $X \perp H_\a$, which by~\eqref{eq:K2c} is equivalent to showing that $K(H_\a,H_\a,X,\m) = 0$, for any $X  \in \ag \cap H_\a^\perp$. This fact will follow from~\eqref{eq:threeagroot} provided the set of dual roots $H_\beta \perp H_\a$ spans the whole subspace $H_\a^\perp$.

The latter fact is immediate for $\mathrm{D}_r, \, r \ge 4$: in the Euclidean space $\br^r = \ag$, the dual roots are $\pm e_i \pm e_j,\, 1 \le i < j \le r$, for some orthonormal basis $\{e_i\}$. Taking $S = \{e_i+e_j \, | \, 1 \le i < j \le r\}$ we see that for any element $e_i + e_j \in S$, the dual roots $e_i-e_j$ and $e_k \pm e_l$ with $\{k,l\} \cap \{i,j\} = \varnothing$ span the whole subspace $(e_i+e_j)^\perp \subset \ag$.

To see that the same fact holds for the root system $\Delta^*$ of type $E_6$, we consider its presentation in the Euclidean space $\br^9$ with an orthonormal basis $\{e_i^a\},\, i, a = 1,2,3$, and with $\ag = (\Span(\sum_{i=1}^{3}e_i^1, \sum_{i=1}^{3}e_i^2, \sum_{i=1}^{3}e_i^3))^\perp$. According to~\cite{VdJ}, we have $\Delta^* = S_1 \cup S_2$, where $S_1=\{e_i^a-e_j^a \,| \, i \ne j\}$ and $S_2=\{\pm \frac13 (v-2(e_i^1+e_j^2+e_k^3))\}$ for $v=\sum_{i,a=1}^{3}e_i^a$. Now the elements of $S_1$ span $\ag$, and for each $H \in S_1$, it is easy to check that the orthogonal complement $\ag \cap H^\perp$ is spanned by the elements of $\Delta^*$ orthogonal to $H$, as required.
	\end{proof}
	
	\subsection{The algebras \texorpdfstring{$\so(n)$, $\su(n)$ and $\spg(n)$}{so(n), su(n) and sp(n)}}
	\label{ss:algebras}
	
	We will use the following bases for the algebras $\so(n)$, $\su(n)$ and $\spg(n)$. We denote by $E_{ij}$ the $n \times n$ matrix whose $(i,j)$-entry is $1$ and whose other entries are zero.
	
	For $\h = \so(n)$, we choose an orthonormal basis $\{e_i\}$ for $\br^n$ and then take a basis for $\so(n)$ consisting of the elements $A_{ij} = E_{ij} - E_{ji}, \, 1 \le i < j\le n$. This basis is orthonormal (up to scaling). For a Cartan subalgebra $\ag \subset \h$, one can choose the subspace $\Span\{A_{2k-1,2k} \, | \, 2k \le n\}$. For any $r=1, \dots, n$, the subalgebra $\so(n-1)_r$ spanned by the elements $A_{ij}$ with $i,j \ne r$ corresponds to the subgroup $\Spin(n-1)_r \subset \Spin(n)$ which is a totally geodesic submanifold.
	
	For $\h = \su(n)$, we choose a Hermitian orthonormal basis $\{e_i\}$ for $\bc^n$ and then take a basis for $\su(n)$ consisting of the elements $A_{ij} = E_{ij} - E_{ji}, \, B_{ij} = \ir (E_{ij} + E_{ji}), \, 1 \le i < j\le n$, and the elements $D_1, \dots, D_{n-1}$ forming a basis for the Cartan subalgebra $\ag = \Span(\diag(\ir \la_1, \dots, \ir \la_n) \, | \, \la_j \in \br, \sum_{j=1}^{n} \la_j = 0)$. For $r=1, \dots, n$, the subalgebra $\su(n-1)_r$ consisting of the matrices of $\su(n)$ annihilating vector $e_r$ corresponds to the totally geodesic subgroup $\SU(n-1)_r \subset \SU(n)$.
	
	For $\h = \spg(n)$, we take a quaternion-Hermitian orthonormal basis $\{e_i\}$ for $\bh^n$ and then take a basis for $\spg(n)$ consisting of the elements $A_{ij} = E_{ij} - E_{ji}, \, B_{ij}(q) = q (E_{ij} + E_{ji}), \, 1 \le i < j \le n$, and the elements $B_{ii}(q) = q E_{ii}, \, 1 \le i \le n$, where $q \in \{\ir,\jr,\kr\}$ is one of the elements of the standard quaternion basis. A particular choice of a Cartan subalgebra can be $\ag = \Span(B_{jj}(\ir) \, | \, j =  1, \dots, n)$. For $r=1, \dots, n$, the subalgebra $\spg(n-1)_r$ consisting of the elements of $\spg(n)$ which annihilate $e_r$ corresponds to the totally geodesic subgroup $\Sp(n-1)_r \subset \Sp(n)$.
	
	\begin{remark} \label{rem:lts}
		The significance of totally geodesic subgroups for our proof comes from the fact that the restriction of a Killing tensor field to a totally geodesic submanifold is again a Killing tensor field. In terms of the tensor $K$, this means that if $K \in \sT(\m)$, then its restriction to any Lie triple system $\m' \subset \m$ is an element of $\sT(\m')$ (recall that complete, totally geodesic submanifolds of a simply connected globally symmetric space $M$ passing through a point $o \in M$ are in one-to-one correspondence with the Lie triple systems $\m' \subset \m = T_oM$). This fact can also be seen directly from equations~\eqref{eq:K2c}, \eqref{eq:K21} and~\eqref{eq:K22} as by definition, a Lie triple system $\m' \subset \m$ is a subspace with the property $R(\m',\m')\m' \subset \m'$.
	\end{remark}
	
We also make the following observations.
	
	\begin{remark} \label{rem:spintoso}
		We will be proving Theorem~\ref{t:SOSUSp} for the group $\Spin(n)$, the simply connected double cover of $\SO(n)$, when $n \ge 3$ (note that Proposition~\ref{p:topslot2} requires the symmetric space $M$ to be simply connected). To deduce the proof for the group $\SO(n)$ from the proof for $\Spin(n)$, it suffices to show that any Killing vector field on $\Spin(n)$ can be projected to a Killing field on $\SO(n)$ under the covering projection. Killing vector fields on $\Spin(n)$ are in one-to-one correspondence with the elements of the Lie algebra $\so(n) \oplus \so(n)$ (direct sum of ideals) of the identity component $\Spin(n) \times \Spin(n)$ of the isometry group of $\Spin(n)$. Given a Killing vector field $\cV$ on $\Spin(n)$ defined by an element $L_1 \oplus L_2 \in \so(n) \oplus \so(n)$, we can construct a one-parameter group of isometries of $\SO(n)$ defined by $(U,t) \mapsto \exp(tL_1) \, U \exp(-tL_2)$, for $U \in \SO(n)$ and $t \in \br$, which in turn defines the Killing vector field $\cV'$ on $\SO(n)$ given by $\cV'(U) = L_1U - UL_2$, where we view $\SO(n)$ as a submanifold in the space $\br^{n^2}$ of $n \times n$ real matrices. Lifting $\cV'$ to $\Spin(n)$ we obtain a Killing vector field $\cV''$. But the Killing vector fields $\cV$ and $\cV''$ on $\Spin(n)$ coincide in a neighbourhood of the identity, and hence coincide everywhere. It follows that any Killing vector field on $\Spin(n)$ is a lift of a Killing vector field on $\SO(n)$.
	\end{remark}

	\begin{remark} \label{rem:low}
		We note that $\Sp(1)=\SU(2)=\Spin(3)$ is isometric to the round sphere $S^3$ and $\Spin(4)$, to the product $S^3 \times S^3$ of the round spheres. We also note the low-dimensional isomorphisms $\spg(2)\cong\so(5)$ and $\su(4)\cong\so(6)$ of the real Lie algebras. Since Killing tensor fields on spaces of constant curvature are decomposable (and since the product cases are treated in~\cite[Theorem~1]{MN2}), it is sufficient to prove Theorem~\ref{t:SOSUSp} for $\Spin(n)$ with $n\geq 5$, for $\SU(n)$ with $n=3$ or $n\geq 5$, and for $\Sp(n)$ with $n\geq 3$.
	\end{remark}

	\section{Proof of Theorem~\ref{t:SOSUSp}}
	\label{s:proof}
	
	By Proposition~\ref{p:topslot2}, we need to prove that for each of the three algebras $\h = \so(n), \, \su(n)$, $\spg(n)$, any tensor $K \in \sT(\h)$ has the form $K_Q$ given in~\eqref{eq:KQpolar}, for a quadratic form $Q$ on $\h$. From now on, we identify $\m$ with $\h$ and write $\sT(\h)$ for $\sT(\m)$.
	
	\subsection{Induction}
	\label{ss:induction}
	
	We argue by induction on $n$. Suppose that for a given $\h=\h(n) = \so(n), \su(n), \spg(n)$, the claim has already been established for all $n < \nu$ for some $\nu \ge 5$, and take $n = \nu$. Given $K \in \sT(\h)$, we construct a certain quadratic form $Q$ on $\h$, and then show that with such a $Q$, equation~\eqref{eq:KQpolar} is satisfied.
	
	First suppose that $\h = \spg(n)$. For $r=1, \dots, n$, we define the subalgebras $\spg(n-1)_r$ as in Subsection~\ref{ss:algebras}. We also use the basis $\{A_{ij}, B_{ij}(q), B_{kk}(q)\}$ introduced there, where $1 \le i < j\le n$ and $q \in \{\ir,\jr,\kr\}$. By Remark~\ref{rem:lts}, the restriction of $K \in \sT(\spg(n))$ to $\spg(n-1)_r$ is an element $K^r \in \sT(\spg(n-1)_r)$, and so by the induction assumption, there exists a quadratic form $Q^r$ on $\spg(n-1)_r$ such that $K^r = K_{Q^r}$. Now take $s \ne r$ and similarly construct a quadratic form $Q^s$ on $\spg(n-1)_s$. Restricting each of the forms $Q^r$ and $Q^s$ to the subalgebra $\spg(n-2)_{rs} = \spg(n-1)_r \cap \spg(n-1)_s$, we obtain two quadratic forms on $\spg(n-2)_{rs}$, for each of which equation~\eqref{eq:KQpolar} is satisfied, for all $X, Y, P, Z \in \spg(n-2)_{rs}$. Then by Lemma~\ref{l:nosyz}, these two restrictions must coincide. Now we define a quadratic form $Q$ on $\spg(n)$ by its coefficients relative to the chosen basis for $\spg(n)$ as follows. We let $C_{ij}, \ i < j$, to denote any of the basis elements $A_{ij}, B_{ij}(q)$, and let $C_{ii}$ to denote any of the basis elements $B_{ii}(q)$. For $i \le j$ and $k \le l$, we take any $r \notin \{i,j,k,l\}$ and set $Q(C_{ij}, C_{kl}) = Q^r(C_{ij}, C_{kl})$. By the previous argument, the form $Q$ is well defined.
	
	The same construction, replacing $C_{ij}$ by $A_{ij}$, works for $\h=\so(n)$ giving us a well-defined quadratic form $Q$ on $\so(n)$. Note that for $n=5$, the required quadratic forms $Q^r$ on $\so(4)_r$ exist by Remark~\ref{rem:Spin4}. The fact that the restrictions of the quadratic forms $Q^r$ and $Q^s$ coincide on $\so(n-2)_{rs}$ follows from assertion~\eqref{it:nosyzSpin} of Lemma~\ref{l:nosyz} when $n=5,6$ and from assertion~\eqref{it:nosyzirred} when $n \ge 7$.
	
	A similar line of argument, with a slight modification, also works for $\h = \su(n)$. Because of the diagonal elements, we should be a little more careful with defining the quadratic form $Q$. In the notation of Subsection~\ref{ss:algebras}, the algebra $\su(n)$ is spanned by the orthonormal off-diagonal matrices $A_{ij}$ and $B_{ij}$, with $1 \le i < j\le n$, and by the linearly dependent diagonal matrices $D_{ij} = \ir (E_{ii}- E_{jj}), \, 1 \le i, j\le n$. For $r=1, \dots, n$, we define the subalgebra $\su(n-1)_r \subset \su(n)$ of the matrices which annihilate the vector $e_r$. By the induction assumption, there exists a quadratic form $Q^r$ on $\su(n-1)_r$ such that $K^r = K_{Q^r}$. We now take $s \ne r$ and similarly construct the corresponding quadratic form $Q^s$ on $\su(n-1)_s$. Their restrictions to the subalgebra $\su(n-2)_{rs} = \su(n-1)_r \cap \su(n-1)_s$ must coincide by Lemma~\ref{l:nosyz}. Next we define a quadratic form $Q$ on $\su(n)$ as follows. For $i < j$ and $k < l$, take any $r \notin \{i,j,k,l\}$ and set $Q(F_{ij}, F_{kl}) = Q^r(F_{ij}, F_{kl})$, where $F_{ab}$ is any of the elements $A_{ab}, B_{ab}$ or $D_{ab}$. Then we extend $Q$ by bilinearity to a quadratic form on the whole algebra $\su(n)$. To verify that the form $Q$ so constructed is well defined, we need to check that our definition agrees with all the linear relations between the elements $D_{ij}$, that is, we need to check that $Q(D_{ij}, F_{lm}) + Q(D_{jk}, F_{lm}) + Q(D_{ki}, F_{lm}) = 0$, for any $5$-tuple $S=(i,j,k,l,m)$. This is satisfied by construction and by the induction assumption when $n-1 \ge 5$ and when $n = 5$, but at least two elements of $S$ are the same. The only case to consider is when $n=5$ and $S'=(1,2,3,4,5)$, say. Then we use the fact that $Q(D_{ij}, F_{45}) = Q(D_{i4}, F_{45}) - Q(D_{j4}, F_{45})$ which holds for $Q$ because it holds for $Q^k$, where $\{i,j,k\} = \{1,2,3\}$.
	
	Thus on every algebra $\h(n) = \so(n), \su(n), \spg(n)$, we have a quadratic form $Q$ such that, by construction, the tensor $K' = K - K_Q \in \sT(\h(n))$ vanishes on any quadruple of elements lying in the same subalgebra $\h(n-1)_r$. We want to prove that $K' = 0$.
	
	In $\h(n)=\spg(n)$, we consider the set $\cC$ of the elements $C_{ij}(q) = q E_{ij} - \overline{q} E_{ji}, \, 1 \le i,j \le n, \, q \in \{1,\ir,\jr,\kr\}$. Clearly, $\cC$ spans $\spg(n)$, and $C_{ij}(q) \in \spg(n-1)_r$, for $r \ne i,j$. Moreover, $[C_{ij}(q), C_{kl}(q')] = 0$ when $\{i,j\} \cap \{k,l\} = \varnothing$. For a quadruple of elements $C_{ab}(q_1), C_{cd}(q_2), C_{ef}(q_3), C_{gh}(q_4)$, denote $S = (a,b,c,d,e,f,g,h)$, and let $|S|$ be the number of distinct elements of $S$. Now $K'(C_{ab}(q_1), C_{cd}(q_2), C_{ef}(q_3), C_{gh}(q_4)) = 0$ by the induction assumption, whenever $|S| < n$. As we always have $|S| \le 8$, this already completes the induction step for $n \ge 9$. When $n = 8$ and $|S| = 8$, then any two elements of the quadruple commute, and the claim of the induction step follows from Lemma~\ref{l:fourcommute}. If $n = 7$ and $|S| = 7$, then in $S$ there is exactly one repeated value, and so there are three elements in the quadruple whose all six subscripts are different. Hence we have, say, $\{a,b\} \cap \{e,f,g,h\} = \varnothing$. But then $[C_{ab}(q_1), C_{ef}(q_3)] = [C_{ab}(q_1), C_{gh}(q_4)] = 0$, and the claim follows by Lemma~\ref{l:threecommute}. Let $n = 6$ and $|S|=6$. Applying Lemma~\ref{l:threecommute} we see that the value of $K'$ is zero on any quadruple, except possibly, on the quadruple $(C_{12}(q_1),C_{36}(q_2),C_{15}(q_3),C_{34}(q_4))$ and on those obtained from it by a permutation of the subscripts $\{1,2,3,4,5,6\}$. For the fixed $q_1,q_2,q_3,q_4$ in this exceptional quadruple, take $x_i,p_i \in \br$ and substitute the vectors $X=x_1 C_{13}(\overline{q_2})+x_2 C_{15}(q_3)+x_3 C_{34}(q_4)$ and $P = p_1 C_{12}(q_1) + p_2 C_{15}(q_3) + p_3 C_{56}(-\overline{q_3})$ in equation~\eqref{eq:K21}. A direct calculation shows that the coefficient of $x_1x_2x_3p_1p_2p_3$ in the resulting equation gives $K'(C_{12}(q_1),C_{36}(q_2),C_{15}(q_3),C_{34}(q_4))=0$, modulo all the terms which vanish by Lemma~\ref{l:threecommute}, which proves the induction step.
	
	The above proof works almost verbatim for the induction step for $\h=\so(n), \, n \ge 6$: one just replaces $C_{ij}(q)$ by $A_{ij}$.
	
	For $\h(n) = \su(n)$, the argument is similar, with some modifications. We consider the set $\cC$ of elements $F_{ij}$, each of which is either $C_{ij}(z) = z E_{ij} - \overline{z} E_{ji},\, z \in \{1,\ir\}$, or $D_{ij} = \ir (E_{ii}-E_{jj})$. Then $\cC$ spans $\su(n)$, and any $F_{ij}$ lies in $\su(n-1)_r$, with $r \ne i,j$. Moreover, $[F_{ij},F_{kl}] = 0$ when $\{i,j\} \cap \{k,l\} = \varnothing$. Repeating the above argument we obtain the induction step for all $n \ge 7$. If $n=6$ and $|S|=6$ for $S=(a,b,c,d,e,f,g,h)$, by Lemma~\ref{l:threecommute} the value of $K'$ is zero on any quadruple $(F_{ab},F_{cd},F_{ef},F_{gh})$, except possibly, on the quadruple $(F_{12},F_{36},F_{15},F_{34})$ and on those obtained from it by a permutation of the subscripts $\{1,2,3,4,5,6\}$. Suppose $F_{ij}=D_{ij}$ for at least one of the elements in this quadruple. Up to a permutation, we can assume that $F_{12}=D_{12}$. Then $K'(F_{12},F_{36},F_{15},F_{34}) = K'(F_{13},F_{36},F_{15},F_{34}) - K'(F_{23},F_{36},F_{15},F_{34}) = 0$, by the induction assumption and by Lemma~\ref{l:threecommute}. Otherwise, we have $K'(C_{12}(z_1),C_{36}(z_2),C_{15}(z_3),C_{34}(z_4)) = 0$, by the same calculation as for $\spg(n)$.
	
	\smallskip
	
	Summarising the above arguments, together with the low-dimensional cases in Remark~\ref{rem:low}, we see that to complete the proof of Theorem~\ref{t:SOSUSp} we need to show that any $K \in \sT(\h)$ has the form $K_Q$ for the following algebras $\h(n)$:
	\begin{equation} \label{eq:remain}
    {
    \renewcommand{\arraystretch}{1.5}
	\begin{array}{|c|c|c|c|c|c|c|}
    \hline
    \h(n) & \so(5) & \su(3) & \su(5) & \spg(3) & \spg(4) & \spg(5) \\
    \hline
    N=\dim \h(n) & 10 & 8 & 24 & 21 & 36 & 55 \\
    \hline
    N_K = \frac{1}{12}N^2 (N^2-1) & 825 & 336 & 27,600 & 16,170 & 139,860 & 762,300 \\
    \hline
    \end{array}
    }
	\end{equation}
where $N_K$ is the dimension of the space of tensors $K$ satisfying the symmetry conditions~\eqref{eq:K2c}, by~\cite[Remark~3]{MN2}.

	This is done in Subsection~\ref{ss:sage} by directly solving the system of linear equations~\eqref{eq:K2c}, \eqref{eq:K21} and~\eqref{eq:K22} for $K$. Recall that to define the quadratic form $Q$ from the quadratic forms $Q^r$ in the induction step, we need $n \ge 5$, and so for $n=3,4$ we will use only Lemma~\ref{l:threecommute} and Remark~\ref{rem:lts}, and some modification of the above induction step.

	\subsection{Computations}
	\label{ss:sage}

The work in this subsection is computer aided.

We need to show that for each algebra $\h(n)$ in~\eqref{eq:remain}, the solution set of equations~\eqref{eq:K2c}, \eqref{eq:K21} and~\eqref{eq:K22} linear in $K$ is $\Span(K_Q \, | \, Q \in \Sym^2 (\h))$, where $K_Q$ is given by~\eqref{eq:KQ}.

By Lemma~\ref{l:nosyz}, the map $Q \mapsto K_Q$ is injective, and so $\dim\Span(K_Q \, | \, Q \in \Sym^2 (\h)) = \frac12 N (N+1)$. It will be therefore sufficient to verify that the corank of the matrix defined by equations~\eqref{eq:K2c}, \eqref{eq:K21} and~\eqref{eq:K22} equals $\frac12 N (N+1)$. While in principle this reduces to computing the rank of a constant real matrix, the ``brute force'' approach is computationally infeasible due to a very large size of that matrix. Indeed, the dimension $N_K$ of the space of tensors $K$ satisfying the symmetry conditions~\eqref{eq:K2c} is given in~\eqref{eq:remain}, and the number of monomials in the components of $X$ and $P$ is $\binom{N+2}{3}^2$ for~\eqref{eq:K21} and $\binom{N+3}{4}^2$ for~\eqref{eq:K22}. These numbers are already large even for the smaller algebras from~\eqref{eq:remain}, and for the algebra $\spg(5)$ are beyond the memory storage for a desktop computer, with calculations taking indefinite time.

To avoid storing a matrix of an enormous size in the computer memory and to be able to perform calculations in reasonable time, we implement the following approach.

\begin{enumerate}[label=(\Roman*),ref=\Roman*]
     \item \label{it:K21only}
     We only use the symmetry conditions~\eqref{eq:K2c} and equations~\eqref{eq:K21}. Equations~\eqref{eq:K2c} and~\eqref{eq:K21} do not in general imply equation~\eqref{eq:K22} (although this is true for rank one spaces~\cite[Corollary~3]{MN2}). Indeed, one can check that the tensor $K_0$ defined by $K_0(X,Y,Z,P)= 2\< X, Y\> \< Z,P\>-\< X,Z\>\< Y,P\>-\< X,P\>\< Y,Z\>$ satisfies both~\eqref{eq:K2c} and~\eqref{eq:K21}, but the difference between the two sides of~\eqref{eq:K22} computed for $K=K_0$ equals $2(\|X\|^2 \|R(X,P)P\|^2-\|P\|^2\|R(P,X)X\|^2)$. This expression is not identically zero already for the symmetric space $S^2 \times \br$, and hence for any symmetric space of which $S^2 \times \br$ is a totally geodesic submanifold, in particular, for any compact irreducible symmetric space of rank at least $2$.

     However, if we find that the corank of the matrix obtained from~\eqref{eq:K2c} and~\eqref{eq:K21} alone is at most $1$ greater than the target one, $\frac12 N (N+1)$, we will be done.

     \item \label{it:Fp}
     All computations below are exact \texttt{SageMath} computations over a finite field $\F_p = \Z/p \Z$, with a large prime $p \equiv 1 \pmod 4$. For each algebra in~\eqref{eq:remain}, we define a basis $\cB$ consisting of matrices whose entries are Gaussian integers, the elements of the ring $\Z[\ir]$. Relative to the inner product $\<X, Y\> = -\frac12 \Tr(XY)$, the basis $\cB$ may not be orthonormal (and not even orthogonal for $\su(3)$ and $\su(5)$). The field $\F_p$ contains an element $\zeta$ with $\zeta^2 = -1$. We define the ring homomorphism $\Xi_p: \Z[\ir] \to \F_p$ by $\Xi_p(\ir) =  \zeta$ and $\Xi_p(a) = a \pmod p$ for $a \in \Z$. With $p$ large enough, the set of matrices $\Xi_p(\cB)$ is still linearly independent, and its Gram matrix relative to the above inner product is non-singular in $\F_p$. We consider equations~\eqref{eq:K2c} and~\eqref{eq:K21} for the elements $K_{IJKL}=K(E_I,E_J,E_K,E_L) \in \F_p$, where $E_I,E_J,E_K,E_L \in \Xi_p(\cB)$. Applying the symmetries~\eqref{eq:K2c} we eliminate some unknowns $K_{IJKL}$ to get exactly $N_K$ of them. Then we obtain a system of linear equations for these unknowns over $\F_p$ by substituting into~\eqref{eq:K21} a large number $N'$ of vectors $X,P \in (\F_p)^N$ all of whose components belong to $\{-2,-1,0,1,2\} \subset \F_p$. The rank $r_p$ of the matrix $M_p$ of that linear system over $\F_p$ is less than or equal to the rank of the original system~\eqref{eq:K21} over $\br$, which gives us the upper bound $N_K - r_p$ for the corank over $\br$. If the latter equals $\frac12 N (N+1) + 1$, the claim follows, as explained in~\eqref{it:K21only}.

     In practice, we take $p=32,749$ and $N'$ in the tens of thousands depending on a particular algebra, and use \texttt{SageMath} packages of finite field arithmetic and sparse matrix rank computation.
\end{enumerate}

These arguments are already sufficient for the algebras $\so(5)$ and $\su(3)$. For $\so(5)$, we choose the basis $\cB$ whose elements are $A_{ij}, \, 1 \le i < j\le 5$, as in Subsection~\ref{ss:algebras}. For $\su(3)$, the basis $\cB$ consists of the six elements $A_{ij} = E_{ij} - E_{ji}, \, B_{ij} = \ir (E_{ij} + E_{ji}), \, 1 \le i < j\le 3$, and two elements $\ir (E_{11} - E_{22})$ and $\ir (E_{22} - E_{33})$, similar to the basis in Subsection~\ref{ss:algebras}. \texttt{SageMath} is extremely effective in finite field arithmetic, and the calculations in both cases take a few minutes on a desktop computer.

By the induction argument in Subsection~\ref{ss:induction}, the claim is already established for all the algebras $\so(n)$, and hence, for $\su(4)$ and $\spg(2)$ -- see Remark~\ref{rem:low}. As the dimension $N_K$ grows rapidly with the dimension of the algebra, for the remaining four algebras in~\eqref{eq:remain} we need a more delicate approach.

\begin{enumerate}[resume,label=(\Roman*),ref=\Roman*]
     \item \label{it:totgeod}
     For every remaining algebra $\h(n)$ in~\eqref{eq:remain}, we can use the fact that for the algebra $\h(n-1)$, the claim is already established. Moreover, when $n = 5$, we can construct a quadratic form $Q$ as in Subsection~\ref{ss:induction} and then define the tensor $K'=K-K_Q$ which is zero on any quadruple of elements belonging to the same subalgebra $\h(5)_r$. This will make a large number of components of $K'$ zero.
\end{enumerate}

Now for the algebra $\su(5)$, we choose the basis $\cB$ consisting of the elements $A_{ij} = E_{ij} - E_{ji}, \, B_{ij} = \ir (E_{ij} + E_{ji}), \, 1 \le i < j\le 5$, and elements $\ir (E_{jj} - E_{j+1,j+1}), \, 1 \le j \le 4$. Following~\eqref{it:totgeod} we construct the tensor $K' \in \sT(\su(5))$ whose value is zero on any quadruple of elements belonging to the subalgebras $\h(5)_r, \, r=1,\dots,5$. Then the \texttt{SageMath} calculation over $\F_p$ implies that $K' = 0$, that is, the rank of the resulting linear system equals the number of the unknowns, from which it follows that $K' = 0$ over $\br$, as well. The calculation took $N'=40 852$ rows.

\smallskip

For the algebras $\spg(n),\, n =3,4,5$, we consider a presentation by $(2n) \times (2n)$ complex matrices coming from the inclusion $\spg(n) \subset \su(2n)$ (rather than by $n \times n$ quaternionic matrices as in Subsection~\ref{ss:algebras}), so that $\spg(n)$ is spanned by the elements
 \begin{equation*}
	X=
	\begin{pmatrix}
		A&B\\
		-\overline B& \overline A
	\end{pmatrix},
	\qquad
	A\in\ug(n),\qquad B\in\Sym^2(n,\bc).
\end{equation*}
We define a basis $\cB$ consisting of the following $(2n) \times (2n)$ matrices over $\Z[\ir]$: $F^1_{ij}=E_{ij}-E_{ji} + E_{n+i,n+j}-E_{n+j,n+i}, \,  F^2_{ij}=\ir (E_{ij}+E_{ji} - E_{n+i,n+j}-E_{n+j,n+i}), \, F^3_{ij}=E_{i,n+j}-E_{n+j,i} - E_{n+i,j}+E_{j,n+i}, \,  F^4_{ij}=\ir (E_{i,n+j}+E_{n+j,i} + E_{n+i,j}+E_{j,n+i})$, for $1 \le i < j \le n$, and $F^2_{ii}=\ir (E_{ii} - E_{n+i,n+i}), \, F^3_{ii}= E_{i,n+i}-E_{n+i,i}, \, F^4_{ii}=\ir (E_{i,n+i} + E_{n+i,i})$, for $1 \le i \le n$. Note that the subalgebra $\spg(n-1)_r \subset \spg(n), \, 1 \le r \le n$, is spanned by the elements $F^a_{ij}, \, a=1,2,3,4$ with $i,j \ne r$.

For $\spg(3)$, we consider the subalgebra $\spg(2)_3$ and construct a quadratic form $Q_3$ on it such that~\eqref{eq:KQpolar} holds for any quadruple of elements from $\spg(2)_3$. Extending $Q_3$ to a quadratic form $Q$ on $\spg(3)$ by taking $Q(F^a_{ij},F^b_{kl}) = 0$ when $3 \in \{i,j,k,l\}$ and defining $K' = K - K_Q \in \sT(\spg(3))$, we need to show that the resulting subspace of tensors $K'$ has dimension $\dim \Sym^2(\spg(3)) - \dim \Sym^2(\spg(2)) = 176$ (in view of Lemma~\ref{l:nosyz}). From Table~\eqref{eq:remain}, there are $16,170-825=15,345$ unknown components of $K'$. Following the approach in~\eqref{it:Fp} we indeed find that the dimension of the subspace obtained by reducing modulo $p$ is at most $176$, as required. The \texttt{SageMath} calculation took $N'=27,445$ sample pairs $(X, P) \in (\{-2,-1,0,1,2\})^{42} \subset (\F_p)^{42}$ in less than an hour.

We use a similar, but refined, approach for $\spg(4)$. For each $r=1,2,3,4$, we construct the corresponding quadratic form $Q_r$ on the subalgebra $\spg(3)_r \subset \spg(4)$. By Lemma~\ref{l:nosyz}, for $r \ne s$, the forms $Q_r$ and $Q_s$ coincide on the pairs of elements belonging to $\spg(2)_{rs} = \spg(3)_r \cap \spg(3)_s$ (note that this argument does not work in the previous case, as $\Sp(1)$ has constant curvature). We can now define a quadratic form $Q$ on $\spg(4)$ by defining its values on the pairs of elements of the basis $\cB$ in such a way that for any $r$, it coincides with the form $Q_r$ on the pairs of elements from $\cB$ belonging to the $\spg(3)_r$, and whose value is zero on the pairs of elements from $\cB$ not belonging to any of $\spg(3)_r$. These latter pairs must be of the form $(F^a_{ij}, F^b_{kl})$, where $\{i,j,k,l\}=\{1,2,3,4\}$ and where $a,b=1,2,3,4$; they span a subspace of dimension $48$ in $\spg(4) \odot \spg(4)$. Replacing the tensor $K$ by the tensor $K'=K-K_Q \in \sT(\spg(4))$ we obtain that $K'$ is zero on any quadruple of elements of $\cB$ lying in the same $Q_r$. Furthermore, from Lemma~\ref{l:threecommute} we obtain that $K'(F^{a_1}_{i_1 j_1}, F^{a_2}_{i_2 j_2}, F^{a_3}_{i_3 j_3}, F^{a_4}_{i_4 j_4}) = 0$, provided the set $\{i_1, j_1\}$ has empty intersection with the set $\{i_2, j_2\} \cup \{i_3, j_3\} \cup \{i_4, j_4\}$. These leaves $56,472$ unknown components of $K'$ before applying~\eqref{eq:K21}. Now we apply~\eqref{it:Fp} taking $N' = 86,313$ sample pairs $(X, P) \in (\{-2, -1, 0, 1, 2\})^{72} \subset (\F_p)^{72}$ to show that the dimension of the solution space over $\F_p$ is at most $48$, as required. The calculation took several hours.

The last algebra to consider is $\spg(5)$. This time, by~\eqref{it:totgeod}, we can construct a quadratic form $Q$ on the whole algebra $\spg(5)$ and the tensor $K'=K-K_Q \in \sT(\spg(5))$ such that the value of $K'$ is zero on any quadruple of elements lying in the same subalgebra $\spg(4)_r \subset \spg(5), \, 1 \le r \le 5$. We need to show that $K'=0$. Similarly to the case of $\spg(4)$, we apply Lemma~\ref{l:threecommute} to kill more components of $K'$ obtaining $33,280$ unknowns (note that this number is smaller than for $\spg(4)$ since now the quadratic form $Q$ is defined on the whole algebra). Applying \texttt{SageMath} $\F_p$ arithmetic as explained in~\eqref{it:Fp} we find that the corank of the matrix obtained from~\eqref{eq:K21} is zero, with $N' = 51,000$ sample pairs $(X, P)$. The calculation took several hours on a desktop computer.

\end{document}